\documentclass[a4paper,12pt,final]{amsart}
\usepackage{times,a4wide,mathrsfs,amssymb,dsfont}

\newcommand{\C}{\mathbb{C}}

\newcommand{\QQ}{\mathbb{Q}}

\newcommand{\PP}{\mathbb{P}}

\newcommand{\OO}{\mathcal O}

\newcommand{\XX}{\mathcal X}
\newcommand{\YY}{\mathcal Y}
\newcommand{\GG}{\mathcal G}

\newcommand{\LLL}{\mathcal L}
\newcommand{\VV}{\mathcal V}
\newcommand{\WW}{\mathcal W}

\newcommand{\codim}{\hbox{codim}}

\newcommand{\wt}{\widetilde}
\newcommand{\ima}{\hbox{Im}}

\newtheorem{theorem}{Theorem}[section]
\newtheorem{claim}[theorem]{Claim}
\newtheorem{lemma}[theorem]{Lemma}

\newtheorem{nonumbering}{Theorem}

\newtheorem{convention}{Conventions}

\theoremstyle{definition}
\newtheorem{remark}[theorem]{Remark}

\newtheorem{nonumberingt}{Acknowledgements}

\begin{document}
\author[Robert Laterveer]
{Robert Laterveer}

\address{Institut de Recherche Math\'ematique Avanc\'ee,
CNRS -- Universit\'e 
de Strasbourg,\
7 Rue Ren\'e Des\-car\-tes, 67084 Strasbourg CEDEX,
FRANCE.}
\email{robert.laterveer@math.unistra.fr}

\title{On the Chow ring of certain hypersurfaces in a Grassmannian}

\begin{abstract} This note is about Pl\"ucker hyperplane sections $X$ of the Grassmannian $\operatorname{Gr}(3,V_{10})$. Inspired by the analogy with cubic fourfolds, we prove that the only non-trivial Chow group of $X$ is generated by Grassmannians of type $\operatorname{Gr}(3,W_{6})$ contained in $X$. We also prove that a certain subring of the Chow ring of $X$ (containing all intersections of positive-codimensional subvarieties) injects into cohomology.
\end{abstract}



\maketitle

\section{Introduction}

 Let $\LLL$ be the Pl\"ucker polarization on the complex Grassmannian $\operatorname{Gr}(3,V_{10})$, and let
   \[ X\in \vert \LLL\vert \]
   be a smooth hypersurface in the linear system of $\LLL$. The Hodge diamond of the $20$-dimensional variety $X$ is
     \[ \begin{array}[c]{ccccccccc}
	&&&&1&&&&\\
	&&&  &2  &  &&&\\
	&&&  &3  &  &&&\\	
	&&&&\ast&&&&\\
	&&&&\vdots&&&&\\
	&&&&\ast&&&&\\
		0&\dots\ \  \dots&	0&1 & 30  &1&0&\dots \ \ \dots&0\\
		&&&&\ast&&&&\\
	&&&&\vdots&&&&\\
		&&&  &\ast&  &&&\\
	&&&  &3  &  &&&\\	
	&&&  &2  &  &&&\\
           &&&&1&&&&\\
	\end{array}\]	
	(Here $\ast$ indicates some unspecified number, and all empty entries are $0$. The Hodge numbers of the vanishing cohomology can be found in \cite[Theorem 1.1]{DV}; alternatively, they can be computed using \cite[Theorem 1.1]{FM}.) 
	
This looks much like the Hodge diamond of a cubic fourfold. To further this analogy,
	Debarre and Voisin \cite{DV} have constructed, for a general such hypersurface $X$, a hyperk\"ahler fourfold $Y$ that is associated (via an Abel--Jacobi isomorphism) to $X$.
	Just as in the famous Beauville--Donagi construction starting from a cubic fourfold \cite{BV}, the hyperk\"ahler fourfolds $Y$ form a $20$-dimensional family, deformation equivalent to the Hilbert square of a K3 surface. The analogy
  \[  \hbox{Pl\"ucker\ hypersurfaces\ in\ $\operatorname{Gr}(3,V_{10})$}  	\ \leftrightsquigarrow    \ \hbox{cubic\ fourfolds} \]
  also exists on the level of derived categories \cite[Section 4.4]{Kuz}.
  	
	In this note we will be interested in the Chow ring $A^\ast(X)_{\QQ}$ of the hypersurface $X$. Using her celebrated method of {\em spread\/} of algebraic cycles in families, Voisin 
	\cite[Theorem 2.4]{V1} (cf. also the proof of theorem \ref{main} below) has already proven a form of the Bloch conjecture for $X$: one has vanishing
	\[  A^i_{hom}(X)_{\QQ}=0\ \ \ \forall\ i\not=11\ \]
	 (where $A^i_{hom}(X)_{\QQ}$ is defined as the kernel of the cycle class map to singular cohomology).
	 This is the analogue of the well-known fact that the only non-trivial Chow group of a cubic fourfold is the Chow group of $1$-cycles.
	 
	 We complete Voisin's result, by describing the only non-trivial Chow group of $X$:
	 
	 \begin{nonumbering}[=theorem \ref{main}] Let $\LLL$ be the Pl\"ucker polarization on $\operatorname{Gr}(3,V_{10})$. Let $X\in\vert \LLL\vert$ be a smooth hypersurface for which the associated hyperk\"ahler fourfold $Y$ is smooth. Then $A^{11}_{hom}(X)_{\QQ}$ is generated by Grassmannians $\operatorname{Gr}(3,W_6)$ contained in $X$.
	 \end{nonumbering}
	 
This is the analogue of the well-known fact that for a cubic fourfold $V\subset\PP^5(\C)$, the Chow group $A^3(V)$ is generated by lines \cite{Par}.	 Theorem \ref{main} is readily proven using the spread method of \cite{V0}, \cite{V1}, \cite{Vo}; as such, theorem \ref{main} could naturally have been included in \cite{V1}.

  The second result of this note concerns the ring structure of the Chow ring of $X$, given by the intersection product:
  
  \begin{nonumbering}[=theorem \ref{main2}] Let $\LLL$ be the Pl\"ucker polarization on $\operatorname{Gr}(3,V_{10})$, and let $X\in\vert \LLL\vert$ be a smooth hypersurface.  
  Let $R^{11}(X)\subset A^{11}(X)_{\QQ}$ be the subgroup containing intersections of two cycles of positive codimension, the Chern class $c_{11}(T_X)$ and the image of the restriction map $A^{11}(\operatorname{Gr}(3,V_{10}))_{\QQ}\to A^{11}(X)_{\QQ}$. The cycle class map induces an injection
    \[ R^{11}(X)\ \hookrightarrow\ H^{22}(X,\QQ)\ .\]
      \end{nonumbering}
      
   This is reminiscent of the famous result about the Chow ring of a K3 surface \cite{BV}. It is also an analogue of the fact that for a cubic fourfold $V$, the subgroup $A^2(V)_{\QQ}\cdot A^1(V)_{\QQ}\subset A^3(V)_{\QQ}$ is one-dimensional. Theorem \ref{main2} suggests that the hypersurfaces $X$ might have a multiplicative Chow--K\"unneth decomposition, in the sense of Shen--Vial \cite{SV}. This seems difficult to establish, however (cf. remark \ref{difficult}).

\vskip0.6cm

\begin{convention} In this note, the word {\sl variety\/} will refer to a reduced irreducible scheme of finite type over $\C$. For a smooth variety $X$, we will denote by $A^j(X)$ the Chow group of codimension $j$ cycles on $X$ 
with $\QQ$--coefficients.

The notations 
$A^j_{hom}(X)$, $A^i_{AJ}(X)$ will be used to indicate the subgroups of 
homologically trivial (resp. Abel--Jacobi trivial) cycles.

For a morphism between smooth varieties $f\colon X\to Y$, we will write $\Gamma_f\in A^\ast(X\times Y)$ for the graph of $f$, and ${}^t \Gamma_f\in A^\ast(Y\times X)$ for the transpose correspondence.


We will write $H^\ast(X)=H^\ast(X,\QQ)$ for singular cohomology with $\QQ$-coefficients.
\end{convention}

\section{Generators for $A^{11}$}
\label{sm1}

\begin{theorem}\label{main} Let $\LLL$ be the Pl\"ucker polarization on $\operatorname{Gr}(3,V_{10})$. Let $X\in\vert \LLL\vert$ be a smooth hypersurface for which there is an associated smooth hyperk\"ahler fourfold $Y$. Then $A^{11}_{hom}(X)$ is generated by the classes of Grassmannians $\operatorname{Gr}(3,W_6)\subset X$ (where $W_6\subset V_{10}$ is a 
six-dimensional vector space).
\end{theorem}

\begin{proof} As mentioned in the introduction, Voisin \cite[Theorem 2.4]{V1} has proven that
  \[ A^i_{hom}(X)=0\ \ \ \forall i>11\ .\]
  Using the Bloch--Srinivas ``decomposition of the diagonal'' method \cite{BS}, \cite[Chapter 3]{Vo} (in the precise form of \cite[Theorem 1.7]{moi}), this implies that 
    \[ \hbox{Niveau}\bigl(A^\ast(X)\bigr)\le 2\ \]
 in the language of \cite{moi}, and also (using \cite[Remark 1.8.1]{moi})
       \[ A^i_{AJ}(X)=0\ \ \ \forall i\not=11\ .\]
   But all intermediate Jacobians of $X$ are trivial (there is no odd-degree cohomology), and so   
  \[ A^i_{hom}(X)=0\ \ \ \forall i\not=11\ .\]
  That is, the $20$-dimensional variety $X$ motivically looks like a surface, and so in particular the Hodge conjecture is true for $X$ \cite[Proposition 2.4]{moi}.

Let
  \[ \XX\ \to\ B \]
  denote the universal family of smooth hypersurfaces in the linear system $\vert\LLL\vert$. The base $B$ is the Zariski open in $\PP(\wedge^3 V_{10}^\ast)$ parametrizing $3$-forms 
  $\sigma$ such that the corresponding hyperplane section
    \[ X_\sigma\ \ \subset\ \operatorname{Gr}(3,V_{10})\ \ \subset\ \PP(\wedge^3 V_{10}) \]
    is smooth.
  
  Let $B^\prime\subset B$ be the Zariski open such that the fibre $X_\sigma$ has an associated hyperk\"ahler fourfold $Y_\sigma$, in the sense of \cite{DV}. That is, $B^\prime$ parametrizes $3$-forms $\sigma$ such that both $X_\sigma$ and
  \[ Y_\sigma:= \bigl\{ W_6\in \operatorname{Gr}(6,V_{10})\ \hbox{such\ that}\   \sigma\vert_{W_6}=0   \bigr\}\ \ \subset\ \operatorname{Gr}(6,V_{10}) \]
are smooth of the expected dimension.

We rely on the spread result of Voisin's, in the following form:

\begin{theorem}[Voisin \cite{V1}]\label{claire} Let 
$\Gamma\in A^{20}(\XX\times_{B^{}} \XX)$ be a relative correspondence with the property that
  \[ (\Gamma\vert_{X_\sigma\times X_\sigma})_\ast H^{11,9}(X_\sigma)=0\ \ \ \hbox{for\ very\ general\ }\sigma\in B^{}\ .\]
 Then
  \[    (\Gamma\vert_{X_\sigma\times X_\sigma})_\ast A^{11}_{hom}(X_\sigma)=0\ \ \ \hbox{for\ all\ }\sigma\in B^{}\ .\]
\end{theorem}

(For basics on the formalism of relative correspondences, cf. \cite[Section 8.1]{MNP}.) Since theorem \ref{claire} is not stated precisely in this form in \cite{V1}, we briefly indicate the proof:

\begin{proof}(of theorem \ref{claire}) The assumption on $\Gamma$ (plus the shape of the Hodge diamond of $X_\sigma$, and the truth of the Hodge conjecture for $X_\sigma$, as shown above) implies that for the very general $\sigma\in B^{}$ there exist closed subvarieties
 $V_\sigma^i,W_\sigma^i\subset X_\sigma$ with $\dim V_\sigma^i+\dim W_\sigma^i=20$, and such that
   \[ \Gamma\vert_{X_\sigma\times X_\sigma}= \sum_{i=1}^s V_\sigma^i\times W_\sigma^i\ \ \ \hbox{in}\ H^{40}(X_\sigma\times X_\sigma)\ .\]
  One has the Noether--Lefschetz property that $H^{20}(X_\sigma,\QQ)_{van}\cap F^{10}=0$ for very general $X_\sigma$ (this is because the Picard number of the very general Debarre--Voisin hyperk\"ahler fourfold is $1$). This implies that all the subvarieties  $V_\sigma^i,W_\sigma^i$ are obtained by restriction from subvarieties of $\operatorname{Gr}(3,V_{10})$, hence they exist universally. (Instead of evoking Noether--Lefschetz, one could also apply Voisin's Hilbert scheme argument \cite[Proposition 3.7]{V0} to obtain that the $V_\sigma^i,W_\sigma^i$ exist universally). That is, there exist closed subvarieties $\VV^i, \WW^i\subset \XX$ with $\codim \VV^i+\codim \WW^i=20$, and a cycle $\delta$ supported on $\cup  \VV^i\times_B \WW^i$,  such that
  \[ \bigl(  \Gamma - \delta)\vert_{X_\sigma\times X_\sigma}=0 \ \ \ \hbox{in}\ H^{40}(X_\sigma\times X_\sigma)\ ,\ \ \ \hbox{for\ very\ general\ }\sigma\in B\ .\] 
  We now define a relative correspondence
  \[ R:=   \Gamma - \delta\ \ \in A^{20}(\XX\times_{B}\XX)\ .\]

 For brevity, from now on let us write $M:=\operatorname{Gr}(3,V_{10})$. Since $M$ has trivial Chow groups (this is true for all Grassmannians, and more generally for {\em linear varieties\/}, cf. \cite[Theorem 3]{Tot}), 
 and the hypersurfaces $X_\sigma$ have non-zero primitive cohomology (indeed $h^{11,9}(X_\sigma)=1$),
 we are in the set--up of \cite{V1}.  
As in loc. cit., we consider the blow-up $\wt{{M}\times{M}}$ of ${M}\times{M}$ along the diagonal, and the quotient morphism $\mu\colon\wt{{M}\times{M}}\to M^{[2]}$ to the Hilbert scheme of length $2$ subschemes. Let $\bar{B}:=\PP H^0({M},\LLL)$
and as in \cite[Lemma 1.3]{V1}, introduce the incidence variety
      \[ I:=\bigl\{ (\sigma,y)\in \bar{B}\times\wt{{M}\times{M}}\ \vert\ s\vert_{\mu(y)}=0\bigr\}\ .\]
Since $\LLL$ is very ample on ${M}$, $I$ has the structure of a projective bundle over $\wt{M\times M}$. 

Next, let us consider 
  \[ f\colon\ \ \wt{\XX\times_B \XX}\ \to\ \XX\times_B \XX\ ,\]
  the blow-up along the relative diagonal $\Delta_\XX$. There is an open inclusion $\wt{\XX\times_B \XX}\subset I$.
  Hence, given our relative correspondence $R\in A^n(\XX\times_B \XX)$ as above, there exists a (non-canonical) cycle
  $\bar{R}\in A^n(I)$ such that
  \[  \bar{R}\vert_{\wt{\XX\times_B \XX}}= f^\ast(R)\ \ \ \hbox{in}\ A^n({\wt{\XX\times_B \XX}})\ .\]
  Hence, we have
  \[ \bar{R}\vert_{\wt{X_\sigma\times X_\sigma}}= \bigl(f^\ast(R)\bigr)\vert_{\wt{X_\sigma\times X_\sigma}}= (f_\sigma)^\ast(R\vert_{{X_\sigma\times X_\sigma}})    =0\ \ \ \hbox{in}\ H^{40}(    \wt{X_\sigma\times X_\sigma})\ ,\]
  for $\sigma\in B$ very general, by assumption on $R$. (Here, as one might guess, the notation 
   \[ f_\sigma\colon \wt{X_\sigma\times X_\sigma}\ \to\ X_\sigma\times X_\sigma\] 
   indicates the blow-up along the diagonal $\Delta_{X_\sigma}$.)
  
  We now apply \cite[Proposition 1.6]{V1} to the cycle $\bar{R}$. The result is that there exists a cycle $\gamma\in A^{20}(M\times M)$ such that there is a rational equivalence
  \[  R\vert_{X_\sigma\times X_\sigma} =     (f_\sigma)_\ast (\bar{R}\vert_{\wt{X_\sigma\times X_\sigma}})=  \gamma\vert_{X_\sigma\times X_\sigma}\ \ \ \hbox{in}\ A^{20}(X_\sigma\times X_\sigma)\ \ \ \forall \sigma\in B\ .\]
  We know that the restriction of $\gamma$ acts as zero on $A^{11}_{hom}(X_\sigma)$. 
  (Indeed, let $\iota\colon X_\sigma\to M$ denote the inclusion, and let $a\in A^{11}_{hom}(X_\sigma)$. With the aid of Lieberman's lemma \cite[Lemma 3.3]{V6}, one finds that
    \[   \bigl((\iota\times\iota)^\ast(\gamma)\bigr){}_\ast(a)=  \iota^\ast \gamma_\ast \iota_\ast(a)\ \ \ \hbox{in}\ A^{11}_{hom}(X_\sigma)\ .\]
    But $\iota_\ast(a)\in A^{12}_{hom}(M)=0$). 
    
    Thus, it follows that
  \[ \bigl(  R\vert_{X_\sigma\times X_\sigma}\bigr){}_\ast=0\colon\ \   A^{11}_{hom}(X_\sigma)\ \to\   A^{11}_{hom}(X_\sigma) \ \ \ \forall \sigma\in B\ .\]
  For any given $\sigma\in B$, one can construct the subvarieties $\VV^i, \WW^i\subset \XX$ in the above argument in such a way that they are in general position with respect to the fibre $X_\sigma$. This implies that the restriction  
  \[   \delta\vert_{X_\sigma\times X_\sigma}\ \ \in A^{20}(X_\sigma\times X_\sigma) \]
  is a {\em completely decomposed cycle\/}, i.e. a cycle supported on a union of subvarieties $V^\sigma_j\times W^\sigma_j\subset   X_\sigma\times X_\sigma$
   with $\codim(V_j^\sigma)+\codim(W_j^\sigma)=20$. But completely decomposed cycles do not act on $A^\ast_{hom}()$ \cite{BS}, and so
   \[      \bigl(  \Gamma\vert_{X_\sigma\times X_\sigma}\bigr){}_\ast   =   \bigl(  (R +\delta){}\vert_{X_\sigma\times X_\sigma}\bigr){}_\ast=0\colon\ \   A^{11}_{hom}(X_\sigma)\ \to\   A^{11}_{hom}(X_\sigma) \ \ \ \forall \sigma\in B\ .\]
   This ends the proof of theorem \ref{claire}.
  \end{proof}

Let us now pick up the thread of the proof of theorem \ref{main}. As in \cite[Section 2]{DV}, for any $3$-form $\sigma\in B^\prime$ let
  \[ G_\sigma:= \Bigl\{  (W_3, W_6)\ \in \operatorname{Gr}(3,V_{10})\times \operatorname{Gr}(6,V_{10})\ \big\vert\ W_3\subset W_6\ ,\ \sigma\vert_{W_6}=0\ \Bigr\} \]
  denote the incidence variety, with projections
     \[ \begin{array}[c]{ccc}
       G_\sigma&\xrightarrow{p_\sigma}& X_\sigma\\
       \ \ \ \ \downarrow{\scriptstyle q_\sigma}&&\\
         \ \  Y_\sigma\ .&&\\
          \end{array}\]
    The fibres of $q_\sigma$ are $9$-dimensional Grassmannians $\operatorname{Gr}(3,W_{6})$.    
    
    Let $\YY\to B^\prime$ denote the universal family of Debarre--Voisin fourfolds (i.e., $\YY\subset \operatorname{Gr}(6,V_{10})\times B^\prime$ is the subvariety of pairs
    $(W_6,\sigma)$ such that $\sigma\vert_{W_6}=0$), and let
    $\GG\to B^\prime$ be the relative version of $G_\sigma$, with projections
     \[ \begin{array}[c]{ccc}
       \GG&\xrightarrow{p}& \XX\\
       \ \ \ \ \downarrow{\scriptstyle q}&&\\
         \ \  \YY\ .&&\\
          \end{array}\]

     We will rely on an Abel--Jacobi type result from \cite{DV}, concerning the {\em vanishing cohomology\/} defined as
     \[ \begin{split}  H^{20}(X_\sigma,\QQ)_{van}&:= \ker\bigl( H^{20}(X_\sigma,\QQ)\ \to\ H^{22}(\operatorname{Gr}(3,V_{10}),\QQ)\bigr)\ ,\\
                          H^2(Y_\sigma,\QQ)_{van}&:=\ker\bigl( H^2(Y_\sigma,\QQ)\ \to\ H^{42}(\operatorname{Gr}(6,V_{10}),\QQ)\bigr)\ .\\
                          \end{split}\]

     \begin{lemma}\label{aj} Let $\sigma\in B^\prime$ be very general. Then there is an isomorphism
     \[ (q_\sigma)_\ast (p_\sigma)^\ast\colon\ \  H^{20}(X_\sigma,\QQ)_{van}\ \xrightarrow{\cong}\ H^2(Y_\sigma,\QQ)_{van}\ . \]     
     The inverse isomorphism is given by
     \[  H^2(Y_\sigma,\QQ)_{van}\ \xrightarrow{\cdot {1\over \mu}\, g^2}\ H^6(Y_\sigma,\QQ)_{van}\ \xrightarrow{(p_\sigma)_\ast (q_\sigma)^\ast}\  H^{20}(X_\sigma,\QQ)_{van}\ .\]
 (Here $\mu\in\QQ$ is some non-zero number independent of $\sigma$, and $g\in A^1(Y_\sigma)$ is the Pl\"ucker polarization.)
 \end{lemma}
 
 \begin{proof} The first part (i.e. the fact that $(q_\sigma)_\ast (p_\sigma)^\ast$ is an isomorphism on the vanishing cohomology) is \cite[Theorem 2.2 and Corollary 2.7]{DV}. 
 For the second part, we observe that the dual map (with respect to cup product)
  \[  (p_\sigma)_\ast (q_\sigma)^\ast\colon\ \  H^6(Y_\sigma,\QQ)_{van}\ \xrightarrow{}\  H^{20}(X_\sigma,\QQ)_{van}  \]
  is also an isomorphism. In particular, using hard Lefschetz, this means that the composition
   \[   H^2(Y_\sigma,\QQ)_{van}\ \xrightarrow{\cdot  g^2}\ H^6(Y_\sigma,\QQ)_{van}\ \xrightarrow{(p_\sigma)_\ast (q_\sigma)^\ast}\  H^{20}(X_\sigma,\QQ)_{van}\   
       \xrightarrow{(q_\sigma)_\ast (p_\sigma)^\ast}\  H^2(Y_\sigma,\QQ)_{van} \]
       is non-zero (and actually an isomorphism). Hence, the assignment
    \[  <\alpha,\beta>_{\rm vv}:= <\alpha, (q_\sigma)_\ast (p_\sigma)^\ast  (p_\sigma)_\ast (q_\sigma)^\ast   (g^2\cdot \beta) >_{Y_\sigma} \]
    defines a polarization on $ H^{2}(Y_\sigma,\QQ)_{van}$. Here, $     <\alpha,\beta>_{Y_\sigma}$ is the Beauville--Bogomolov form.   
     However, as explained in \cite[Proof of Lemma 2.2]{V1}, for very general $\sigma$ the Hodge structure on  $ H^{2}(Y_\sigma,\QQ)_{van}$ is simple, and admits a unique polarization up to a coefficient. That is, there exists a non-zero number $\mu\in\QQ$ such that
    \[        <\alpha,\beta>_{\rm vv}  =\mu <\alpha,\beta>_{Y_\sigma} \ .\]
    The Beauville--Bogomolov form being non-degenerate, this proves that
    \[  (q_\sigma)_\ast (p_\sigma)^\ast  (p_\sigma)_\ast (q_\sigma)^\ast   (g^2\cdot \beta)=\mu \,\beta\ \ \ \forall \beta\ \in\ H^{2}(Y_\sigma,\QQ)_{van}\ .\]
    Reasoning likewise starting from $  H^{20}(X_\sigma,\QQ)_{van}$ (now using the cup product instead of the Beauville--Bogomolov form), we find that the other composition is also the identity.    
    
   Finally, the fact that the constant $\mu$ is the same for all fibres $X_\sigma$ is because the map in cohomology $H^2(Y_\sigma,\QQ)_{van}\to H^{20}(X_\sigma,\QQ)_{van}$ is locally constant in the family. 
    \end{proof}

Let us define the relative correspondence
  \[ \Gamma:= \mu \Delta_\XX -    \Gamma_p\circ {}^t \Gamma_q \circ \Gamma_{g^2} \circ \Gamma_q\circ {}^t \Gamma_p   \ \ \in\ A^{20}(\XX\times_{B^\prime}\XX)\ , \]
  where $\Gamma_{g^2}\in A^6(\YY\times_{B^\prime}\YY)$ is the correspondence acting fibrewise as intersection with two Pl\"ucker hyperplanes.
  Lemma \ref{aj} implies that
  \[ (\Gamma\vert_{X_\sigma\times X_\sigma})_\ast  H^{20}(X_\sigma,\QQ)_{van}=0\ \ \ \hbox{for\ very\ general\ }\sigma\in B^\prime\ .\]
  That is, the relative correspondence $\Gamma$
  satisfies the assumption of theorem \ref{claire}. Thanks to theorem \ref{claire}, we thus conclude that
  \[    (\Gamma\vert_{X_\sigma\times X_\sigma})_\ast   A^{11}_{hom}(X_\sigma)=0\ \ \ \forall \sigma\in B\ .\]
  Unraveling the definition of $\Gamma$, this means in particular that there is a surjection
   \[   (p_\sigma)_\ast (q_\sigma)^\ast\colon\ \      A^4_{hom}(Y_\sigma)\  \twoheadrightarrow\ A^{11}_{hom}(X_\sigma) \ \ \ \forall \sigma\in B^\prime\ .\]    
   As we have seen, for any point $y\in Y_\sigma$ the fibre $(q_\sigma)^{-1}(y)$ is a $9$-dimensional Grassmannian $\operatorname{Gr}(3,W_6)$ such that the $3$-form $\sigma$ vanishes on $W_6$. Such a Grassmannian is contained in the hypersurface $X_\sigma$, and so 
   \[ (p_\sigma)_\ast (q_\sigma)^\ast   (y) = \operatorname{Gr}(3,W_6)\ \ \ \hbox{in}\ A^{11}(X_\sigma)\ \ \ \forall y\in Y_\sigma\ .\]
   The theorem is proven.   
    \end{proof}

\begin{remark} The above argument actually shows that
  \[ A^{11}_{hom}(X_\sigma) \ \xrightarrow{ (q_\sigma)_\ast (p_\sigma)^\ast}\ A^2_{hom}(Y_\sigma)\ \xrightarrow{\cdot g^2}\ A^4_{hom}(Y_\sigma)\ \xrightarrow{ (p_\sigma)_\ast 
  (q_\sigma)^\ast}\ A^{11}_{hom}(X_\sigma) \]
  is a non-zero multiple of the identity, for any $\sigma\in B^\prime$. This is very much reminiscent of cubic fourfolds and their Fano varieties of lines \cite{BD}, \cite{SV}.
  Inspired by this analogy, it is tempting to ask the following: can one somehow prove that
    \[ \ima \bigl( A^{11}_{}(X_\sigma)\to A^4_{}(Y_\sigma)\bigr) \]
    is the same as the subgroup of $0$-cycles supported on a uniruled divisor ?
  \end{remark}

\section{An injectivity result}

\begin{theorem}\label{main2} Let $\LLL$ be the Pl\"ucker polarization on $\operatorname{Gr}(3,V_{10})$, and let $X\in\vert \LLL\vert$ be a smooth hypersurface.  
  Let $R^{11}(X)\subset A^{11}(X)_{\QQ}$ be the subgroup containing intersections of two cycles of positive codimension, the Chern class $c_{11}(T_X)$ and the image of the restriction map $A^{11}(\operatorname{Gr}(3,V_{10}))_{}\to A^{11}(X)_{}$. The cycle class map induces an injection
    \[ R^{11}(X)\ \hookrightarrow\ H^{22}(X,\QQ)\ .\]
    \end{theorem}

In order to prove theorem \ref{main2}, we first establish a ``generalized Franchetta conjecture'' type of statement (for more on the generalized Franchetta conjecture, cf. \cite{OG}, \cite{PSY}, \cite{FLV}):

\begin{theorem}\label{gfc} Let $\XX\to B$ denote the universal family of Pl\"ucker hyperplanes in $\operatorname{Gr}(3,V_{10})$ (as in section \ref{sm1}).
Let $\Psi\in A^{11}(\XX)$ be such that
  \[ \Psi\vert_{X_\sigma}=0\ \ \hbox{in}\ H^{22}(X_\sigma)\ \ \ \forall \sigma\in B\ .\]
  Then
   \[ \Psi\vert_{X_\sigma}=0\ \ \hbox{in}\ A^{11}(X_\sigma)\ \ \ \forall \sigma\in B\ .\]
\end{theorem}

\begin{proof} This is a two-step argument:

\begin{claim}\label{c1} There is equality
  \[ \ima\bigl( A^{11}(\XX)\to A^{11}(X_\sigma)\bigr)= \ima\bigl(  A^{11}(\operatorname{Gr}(3,V_{10}))_{}\to A^{11}(X_\sigma)\bigr)\ \ \ \forall \sigma\in B\ .\]
 \end{claim}
 
 \begin{claim}\label{c2} Restriction of the cycle class map induces an injection
  \[   \ima\bigl(  A^{11}(\operatorname{Gr}(3,V_{10}))_{}\to A^{11}(X_\sigma)\bigr)\ \hookrightarrow\ H^{22}(X_\sigma)  \ \ \ \forall \sigma\in B\ .\]
 \end{claim}
 
Clearly, the combination of these two claims proves theorem \ref{gfc}. 
 To prove claim \ref{c1}, let $\bar{B}:=\PP H^0( \operatorname{Gr}(3,V_{10}),\LLL)$ and let
    \[ \begin{array}[c]{ccc}
       \bar{\XX}&\xrightarrow{\pi}& \operatorname{Gr}(3,V_{10})  \\
       \ \ \ \ \downarrow{\scriptstyle \phi}&&\\
         \ \  \bar{B}\ &&\\
          \end{array}\]    
          denote the universal hyperplane (including the singular hyperplanes). The morphism $\pi$ is a projective bundle, and so any $\Psi\in A^{11}(\bar{\XX})$ can be written
          \[ \Psi=    \sum_\ell \pi^\ast( a_\ell)   \cdot \phi^\ast(h^\ell)\ \ \ \hbox{in}\ A^{11}(\bar{\XX})\ ,\]
         where $a_\ell\in A^{11-\ell}( \operatorname{Gr}(3,V_{10}))$ and $h:=c_1(\OO_{\bar{B}}(1))\in A^1(\bar{B})$.
      For any $\sigma\in B$, the restriction of $\phi^\ast(h)$ to the fibre $X_\sigma$ vanishes, and so
      \[ \Psi\vert_{X_\sigma} = a_0\vert_{X_\sigma}\ \ \ \hbox{in}\ A^{11}(X_\sigma)\ ,\]
      which establishes claim \ref{c1}.
      
  Let us prove claim \ref{c2}. For any given $\sigma\in B$, let $\iota\colon X_\sigma\to  \operatorname{Gr}(3,V_{10})$ denote the inclusion morphism. We know that
   \[  \iota_\ast \iota^\ast\colon\ \ A^j( \operatorname{Gr}(3,V_{10}))\ \to\   A^{j+1}( \operatorname{Gr}(3,V_{10}))  \]
   equals multiplication by the ample class $c_1(\LLL)\in A^1( \operatorname{Gr}(3,V_{10}))$.
   Now let
   \[ b\in A^{11}( \operatorname{Gr}(3,V_{10})) \]
   be such that the restriction $\iota^\ast(b)\in A^{11}(X_\sigma)$ is homologically trivial. Then we have that also
   \[ b\cdot c_1(\LLL)= \iota_\ast \iota^\ast(b) =0\ \ \ \hbox{in}\ H^{24}( \operatorname{Gr}(3,V_{10}))=A^{12}( \operatorname{Gr}(3,V_{10}))\ .\]
   To conclude that $b=0$, it suffices to show that
   \[ \cdot c_1(\LLL)\colon\ \ A^{11}( \operatorname{Gr}(3,V_{10}))\ \to\ A^{12}(  \operatorname{Gr}(3,V_{10})) \]
   is injective (and hence, by hard Lefschetz, an isomorphism). By hard Lefschetz, this is equivalent to showing that
       \[ \cdot c_1(\LLL)\colon\ \ A^{9}( \operatorname{Gr}(3,V_{10}))\ \to\ A^{10}(  \operatorname{Gr}(3,V_{10})) \] 
      is surjective (hence an isomorphism). 
    
    According to \cite[Theorem 5.26]{EH}, the Chow ring of the Grassmannian is of the form
    \[ A^\ast (\operatorname{Gr}(3,V_{10}))   =  \QQ[c_1,c_2,c_3]/ I\ ,   \]
    where $c_j\in A^j (\operatorname{Gr}(3,V_{10}))$ are Chern classes of the universal subbundle, and $I$ is a certain complete intersection ideal generated by the $3$ relations 
      \[   \begin{split} &c_1^8+ 7c_1^6c_2 + 15 c_1^4c_2^2+10 c_1^2c_2^3 +\cdots+3c_2c_3^2\ ,\\
                                 & c_1^9+ 8c_1^7c_2+21c_1^5 c_2^2 +  20 c_1^3c_2^3    +\cdots + c_3^3\ ,\\
                                 &c_1^{10}+9c_1^8c_2+  28 c_1^6 c_2^2 +  35 c_1^4 c_2^3     +\cdots + 4c_1c_3^3\ \ ,\\
        \end{split}\]
      in degree $8, 9, 10$.
    With the aid of the relations in $I$, we find that
    \[ A^{10}( \operatorname{Gr}(3,V_{10})) =\QQ[ c_1^{10}, c_1^8c_2, c_1^6 c_2^2,c_1^4c_2^3,c_1^7c_3,c_1^5c_2c_3,c_1^4c_3^2,c_1^3c_2^2c_3,c_1^2c_2c_3^2, c_1c_2^3c_3]
    \]
    is $10$-dimensional (the classes $c_1^2 c_2^4, c_2^5$ are eliminated thanks to the relation in degree $8$ containing $c_2^4$; the class $c_1 c_3^3$ is eliminated thanks to the relation in degree $9$; the class $c_2^2 c_3^2$ is eliminated thanks to the relation in degree $10$).
   Inspecting this description of $A^{10}( \operatorname{Gr}(3,V_{10}))$, we observe that the inclusion 
   \[ c_1\cdot A^9( \operatorname{Gr}(3,V_{10})) \ \subset \ A^{10}( \operatorname{Gr}(3,V_{10})) \]
   is an equality. Since $c_1$ is proportional to $c_1(\LLL)$, this proves claim \ref{c2}.                     
\end{proof}

It remains to prove theorem \ref{main2}:

\begin{proof}(of theorem \ref{main2}) Clearly, the Chern class is universally defined: for any $\sigma\in B$, we have
  \[ c_{11}(T_{X_\sigma})= c_{11}(T_{\XX/B})\vert_{X_\sigma}\ .\]
  Also, the image
  \[ \ima \bigl( A^{11}(\operatorname{Gr}(3,V_{10}))_{}\to A^{11}(X_\sigma)\bigr)  \]
  consists of universally defined cycles. (For a given $a\in A^{11}(\operatorname{Gr}(3,V_{10}))$, the relative cycle
  \[   (a\times B)\vert_\XX\ \ \in\ A^{11}(\XX) \]
  does the job.)
  
  Likewise, for any $j<10$ the fact that $A^j_{hom}(X_\sigma)=0$, combined with weak Lefschetz in cohomology, implies that
  \[ A^j(X_\sigma)=\ima\bigl(  A^{j}(\operatorname{Gr}(3,V_{10}))_{}\to A^{j}(X_\sigma)\bigr)  \ ,\]
  and so $ A^j(X_\sigma)$ consists of universally defined cycles for $j<10$. In particular, all intersections
  \[   A^j(X_\sigma)\cdot A^{11-j}(X_\sigma)\ \ \subset\ A^{11}(X_\sigma)\ ,\ \ \ 1<j<10 \]
  consist of universally defined cycles.
  
  It remains to make sense of intersections
   \[ A^{10}(X_\sigma)\cdot A^1(X_\sigma)\ \ \subset\ A^{11}(X_\sigma)\ .\]  
  To this end, we note that $A^1(X_\sigma)$ is $1$-dimensional, generated by the restriction $g$ of the Pl\"ucker line bundle $\LLL$. 
  Let $\iota\colon X_\sigma\to \operatorname{Gr}(3,V_{10})$ denote the inclusion. The normal bundle formula implies that
  \[  a\cdot g=  \iota^\ast \iota_\ast (a)\ \ \ \hbox{in}\ A^{11}(X_\sigma)\ \ \ \forall\ a\in\ A^{10}(X_\sigma)\ .\]
  It follows that
  \[ A^{10}(X_\sigma)\cdot A^1(X_\sigma)\ \ \subset\ \ima   \bigl( A^{11}(\operatorname{Gr}(3,V_{10}))_{}\xrightarrow{\iota^\ast} A^{11}(X_\sigma)\bigr)  \]
  also consists of universally defined cycles.
  
  In conclusion, we have shown that $R^{11}(X_\sigma)$ consists of universally defined cycles, and so theorem \ref{main2} is a corollary of theorem \ref{gfc}.
    \end{proof}

\begin{remark} There are more cycle classes that can be put in the subgroup $R^{11}(X)$ of theorem \ref{main2}. For instance, let $Y_\sigma$ be the hyperk\"ahler fourfold associated to $X=X_\sigma$, and assume $Y_\sigma$ is smooth. Then (as we have seen above) the class 
  \[ (p_\sigma)_\ast(q_\sigma)^\ast c_4(T_{Y_\sigma})\ \ \in\ A^{11}(X) \]
  is universally defined, hence it can be added to the subgroup $R^{11}(X)$ of theorem \ref{main2}. 
 \end{remark}

\begin{remark}\label{difficult} Theorem \ref{main2} is an indication that perhaps the hypersurfaces $X\subset \operatorname{Gr}(3,V_{10})$ have a {\em multiplicative Chow--K\"unneth decomposition\/}, in the sense of \cite[Chapter 8]{SV}. Unfortunately, establishing this seems difficult; one would need something like theorem \ref{gfc} for
  \[ A^{40}(\XX\times_B \XX\times_B \XX)\ .\]
\end{remark}

\vskip1cm
\begin{nonumberingt} Thanks to my mythical colleague Gilberto Kiwi for inspiring conversations, and thanks to a zealous referee for constructive comments.

\end{nonumberingt}

\end{document}